\documentclass{amsart}

\usepackage{amssymb, amsmath, amsthm, amsfonts,bbm}
\usepackage{mathrsfs}
\usepackage{amscd}
\usepackage[active]{srcltx}
\usepackage{verbatim}

\allowdisplaybreaks

\usepackage{newsymbol}
\let\mathcal \undefined
\def\mathcal{\mathscr}

\let\emptyset \undefined
\let\ge       \undefined
\let\le       \undefined
\newsymbol\le          1336  \let\leq\le
\newsymbol\ge          133E  
\newsymbol\emptyset    203F
\newsymbol\notle       230A
\newsymbol\notge       230B

\theoremstyle{plain}
\newtheorem{theorem}{Theorem}[section]
\theoremstyle{remark}
\newtheorem{remark}[theorem]{Remark}
\newtheorem{example}[theorem]{Example}

\theoremstyle{plain}
\newtheorem{corollary}[theorem]{Corollary}
\newtheorem{lemma}[theorem]{Lemma}
\newtheorem{proposition}[theorem]{Proposition}
\newtheorem{definition}[theorem]{Definition}

\numberwithin{equation}{section}

\def\R{{\mathbb R}}

\newcommand{\E}{{\mathbb E}}
\renewcommand{\P}{{\mathbb P}}
\newcommand{\F}{{\mathscr F}}

\renewcommand{\H}{H}
\newcommand{\HH}{{\mathscr H}}

\renewcommand{\a}{\alpha}
\renewcommand{\b}{\beta}
\newcommand{\g}{\gamma}
\renewcommand{\d}{\delta}
\newcommand{\e}{\varepsilon}
\newcommand{\la}{\lambda}

\newcommand{\Om}{\Omega}

\newcommand{\beq}{\begin{equation}}
\newcommand{\eeq}{\end{equation}}
\newcommand{\bal}{\begin{aligned}}
\newcommand{\eal}{\end{aligned}}
\newcommand{\ben}{\begin{enumerate}}
\newcommand{\een}{\end{enumerate}}
\newcommand{\bit}{\begin{itemize}}
\newcommand{\eit}{\end{itemize}}

\newcommand{\bth}{\begin{theorem}}
\renewcommand{\eth}{\end{theorem}}
\newcommand{\bpr}{\begin{proposition}}
\newcommand{\epr}{\end{proposition}}
\newcommand{\ble}{\begin{lemma}}
\newcommand{\ele}{\end{lemma}}
\newcommand{\bpf}{\begin{proof}}
\newcommand{\epf}{\end{proof}}
\newcommand{\bex}{\begin{example}}
\newcommand{\eex}{\end{example}}
\newcommand{\bre}{\begin{example}}
\newcommand{\ere}{\end{example}}

\newcommand{\calL}{{\mathcal L}}
\newcommand{\n}{\Vert}
\newcommand{\one}{\mathbbm{1}}
\newcommand{\embed}{\hookrightarrow}
\newcommand{\s}{^*}
\newcommand{\lb}{\langle}
\newcommand{\rb}{\rangle}

\newcommand{\Ha}{H_{0+}^{\a}}

\newcommand{\Hb}{H_{0+}^{\b+\frac12}}

\newcommand{\HbbT}{H_{T-}^{\frac12-\b}}

\newcommand{\HbbTH}{H_{T-}^{\frac12-\b}(\H)}

\newcommand{\IbbT}{I_{T-}^{\frac12-\b}{}}

\newcommand{\Gb}{\Gamma(\b+\frac12)}

\newcommand{\Wb}{W^\b}
\newcommand{\WbH}{W_\H^\b}

\newcommand{\OO}{\mathcal O}
\newcommand{\wt}{\widetilde}

\begin{document}

\title[Equations driven by Liouville fBm]
{Stochastic evolution equations driven by Liouville fractional
Brownian motion}

\author{Zdzis{\l}aw Brze\'zniak}
\address{Department of Mathematics
\\University of York
\\York YO10 5DD
\\United Kingdom}
\email{zb500@york.ac.uk}

\author{Jan van Neerven}
\address{Delft Institute of Applied Mathematics
\\Delft University of Technology
\\P.O. Box 5031, 2600 GA Delft
\\The Netherlands}
\email{J.M.A.M.vanNeerven@tudelft.nl}

\author{Donna Salopek}
\address{School of Mathematics and Statistics
\\University of New South Wales
\\Sydney, NSW, 2052
\\Australia}
\email{dm.salopek@unsw.edu.au}



\date\today
\begin{abstract} Let $\H$ be a Hilbert space and $E$ a Banach space. We set up a
theory of stochastic integration of $\calL(\H,E)$-valued functions with 
respect to $\H$-cylindrical Liouville fractional Brownian motions with arbitrary
Hurst parameter $0<\b<1$. For $0<\b<\frac12$ we show that a function
$\Phi:(0,T)\to \calL(\H,E)$ is stochastically integrable with respect to an
$\H$-cylindrical Liouville 
fractional Brownian motion if and only if it is stochastically integrable with
respect to an $\H$-cylindrical fractional Brownian motion.

We apply our results to stochastic evolution equations 
$$ dU(t) = AU(t)\,dt + B\, dW_\H^\beta(t)$$ 
driven by an $\H$-cylindrical Liouville fractional Brownian motion, 
and prove existence, uniqueness and space-time regularity of mild solutions
under various assumptions on the Banach space $E$, the operators $A: \mathscr{D}(A)\to E$
and $B:H\to E$, and the Hurst parameter $\b$.  

As an application it is shown that second-order parabolic SPDEs
on bound\-ed domains in $\R^d$, driven by space-time noise which is white in space
and Liouville fractional
in time, admit a mild solution if $\frac{d}{4}<\b<1$.
\end{abstract}

\thanks{The first named author thanks Eurandom for kind hospitality. 
The second named author gratefully acknowledges financial support by
VICI subsidy 639.033.604 of the
Netherlands Organization for Scientific Research (NWO)}

\maketitle

\section{Introduction}\label{sec:intr}

Since the pioneering paper of Mandelbrot and Van Ness \cite{MVNess},  fractional
Brownian motion (fBm) has been proposed as a model to a variety of phenomena in
population dynamics (e.g. \cite{Jg02, J}), random long-time influences in
climate systems (e.g. \cite{Palmer2005, Palmer}), mathematical finance (e.g.
\cite{GRA_09,
GP, Ohashi}, \cite{book}, and the references therein),  random dynamical systems
(e.g. \cite{MS04}), and telecommunications (e.g. \cite{LTWW, w2}). 
This has motivated many studies of
stochastic partial differential equations driven by fBm, among them \cite{GA,
Cai, DMP1,GRA_09, HAIR-OH, Hu,HOZ,DMP2,TTV}.
Following the approach of Da Prato and Zabczyk \cite{DaPZab},
such equations may often be formulated as abstract stochastic evolution
equations on an infinite-dimensional state space. This naturally leads to the
problem of defining a stochastic integral with respect to cylindrical fBm in
such spaces. This problem has been considered by various authors, among them
Duncan, Mas\-lowski, Pasik-Duncan \cite{DMP1} (in Hilbert spaces
for cylindrical fBm with Hurst parameter $0<\b<\frac12$)
and Duncan, Maslowski, Pasik-Duncan \cite{DMP2} (in Hilbert spaces
for fBm with Hurst parameter $\frac12<\b<1$). The stochastic integral constructed in
these papers was used to prove existence of mild solutions
for stochastic abstract Cauchy problems of the form
$$ dU(t) = AU(t)\, + B\, dW_\H^\b(t),\quad U(0) = u_0,$$
where $A$ is the generator of a $C_0$-semigroup on a Hilbert space $E$,
$B$ is a bounded operator from another Hilbert space $\H$ to $E$, and
$W_\H^\b$ is an $\H$-cylindrical fBm
with Hurst parameter $\b$ (in contrast with the literature on fBm, but in line with
the literature of SPDEs, we use the letter $H$ for the Hilbert space associated 
with the cylindrical noise). 
Among other things, for $\frac12<\b<1$ it was shown that a mild solution
always exists if $B$ is a Hilbert-Schmidt operator, and
for $0<\b<\frac12$ the same conclusion holds if one assumes that the semigroup
generated by $A$ is analytic.

The purpose of this paper is to prove analogues of the above-mentioned results
for cylindrical Liouville fBm and to extend the setting to Banach spaces $E$.
Stochastic integration with respect to Liouville fBm turns out to be both
simpler
and more symmetric with respect to the choice of the Hurst parameter below or above
the critical value $\b = \frac12$. In many respects this allows a unified
treatment of both cases. For $0<\b<\frac12$ it turns out that an operator-valued
function $\Phi:(0,T)\to \calL(\H,E)$ is stochastically integrable with respect to
an $\H$-cylindrical fBm if and only if it is stochastically integrable with
respect to an $\H$-cylindrical Liouville fBm.

Our theory is applied to stochastic evolution equations driven by an
$\H$-cylindrical fBm. We show that second-order parabolic SPDEs
on bounded domains in $\R^d$, driven by space-time noise which is white in space
and Liouville fractional in time, admit a mild solution if $\frac{d}{4}<\b<1$.

\medskip
We conclude this introduction with a brief comparison of our results with 
the existing literature. In \cite{TTV} the authors study stochastic evolution 
equations driven by additive cylindrical fBm in a Hilbert space framework for  
self-adjoint operators $A$. When applied to the stochastic heat equation, their 
regularity results appear to be weaker than ours. This seems to be an intrinsic 
feature of the fact that the method is limited to the Hilbert space framework;
in our Banach space framework we are able to use $L^p$-techniques.

The method of Young integrals employed in \cite{GLT_2006} leads to the same 
regularity results as ours in the case of one-dimensional stochastic heat equation. 
The approach taken in this paper is purely pathwise while ours is stochastic. 

Let us finally mention that semilinear stochastic evolution equations in Hilbert 
spaces driven by multiplicative cylindrical fBm have been studied in  
\cite{Masl+Nua_2003} for Hurst parameter $\frac12<\b<1$. We believe that 
the results obtained here can be extended to this class of equations in a Banach space 
framework by following the approach of, e.g., \cite{NVW2}.

\section{Fractional integration spaces}\label{sec:frac}

For $\a>0$ the {\em left Liouville fractional integral} and the
{\em right Liouville fractional integral} of order $\a$ of a
function $f\in L^2(a,b)$ are defined by
$$
\begin{aligned}
(I_{a+}^\a f)(t) &:=  \frac1{\Gamma(\a)} \int_a^t (t-s)^{\a-1} f(s)\,ds, \quad t
\in [a,b]\\
(I_{b-}^\a f)(t) &:=  \frac1{\Gamma(\a)} \int_t^b (s-t)^{\a-1} f(s)\,ds, \quad t
\in [a,b].
\end{aligned}
$$ By Young's inequality, the functions $I_{a+}^\a f$ and
$I_{b-}^\a f$ belong to $L^2(a,b)$. The
operators $I_{a+}^\a$ and $I_{b-}^\a$ are bounded and injective on $L^2(a,b)$,
with dense ranges denoted by
$$
H_{a+}^{\a}(a,b):=I_{a+}^{\a}(L^2(a,b)), \qquad
H_{b-}^{\a}(a,b):=I_{b-}^{\a}(L^2(a,b)).
$$
These spaces are
Hilbert space with respect to the norms
$$
\n I_{a+}^\a f \n_{H_{a+}^{\a}} := \n f\n_{L^2(a,b)}, \qquad \n
I_{b-}^\a f \n_{H_{b-}^{\a}}  := \n f\n_{L^2(a,b)}.
$$
We have continuous inclusions
$$ H_{a+}^{\a}(a,b)\embed L^2(a,b), \qquad  H_{b-}^{\a}(a,b)\embed L^2(a,b).$$
The following simple observation will be used frequently.

\begin{lemma}\label{lem:properties} Let $\a>0$ and $a<x<b$.

\begin{enumerate}
\item {\rm (Restriction \ with \ left \ boundary\  condition)}\ If\ $f\in
H_{a+}^{\a}(a,b)$,\ then\ $f|_{(a,x)}\in
H_{a+}^{\a}(a,x)$ and $$ \n f|_{(a,x)}\n_{H_{a+}^{\a}(a,x)}
\le \n f\n_{H_{a+}^{\a}(a,b)}.$$

\item {\rm (Extension with right boundary condition)}\ If $f\in
H_{x-}^{\a}(a,x)$ and we define $f_x(s) =
f(s)$ for $s\in (a,x)$ and $f_x(s)= 0$ otherwise, then
$f_x\in H_{b-}^{\a}(a,b)$ and $$ \n
f_x\n_{H_{b-}^{\a}(a,b)} = \n f \n_{H_{x-}^{\a}(a,x)}.$$

\item {\rm (Reflection)}\ We have $f\in H_{a+}^{\a}(a,b)$ if and  only if
$\check
f\in H_{b-}^{\a}(a,b)$, where $\check f(t):= f(b-(t-a))$, and in
this situation we have $$ \n f\n_{H_{a+}^{\a}(a,b)} = \n \check f
\n_{H_{b-}^{\a}(a,b)}.$$
\end{enumerate}
\end{lemma}
\begin{proof}
We only prove (1); the proofs of (2) and (3) are similar.  By
assumption, $f = I_{a+}^\a g$ for some $g\in L^2(a,b)$. Clearly
for $s\in (a,x)$ we have $( I_{a+}^\a (g|_{(a,x)}))(s)
= f(s),$ where by slight abuse of notation we also use the
notation $I_{a+}^\a$ for the fractional integration operator
acting on $L^2(a,x)$. It follows that $f|_{(a,x)}\in
H_{a+}^{\a}(a,x)$ and $ I_{a+}^\a (g|_{(a,x)}) = f|_{(a,x)}$.
Moreover, $ \n f|_{(a,x)}\n_{H_{a_+}^{\a}(a,x)} = \n g|_{(a,x)}
\n_{L^2(a,x)} \le\n g\n_{L^2(a,x)} = \n f\n_{H_{a+}^{\a}(a,b)}.$
\end{proof}

The following result is less elementary; for a proof we refer to
\cite[Chapter 3, Section 13.3]{Samko}.

\begin{lemma}\label{lem:restr}
Let $0<\a<\frac12$. Then we have $H_{a+}^{\a}(a,b) =
H_{b-}^{\a}(a,b)$ with equivalent norms. As a consequence,
for all $a<x<b$ there is a constant $C_{\a,x}$ such that:

\begin{enumerate}
\item {\rm (Restriction  with right boundary condition)}\ If $f\in
H_{b-}^{\a}(a,b)$, then
$f|_{(a,x)}\in H_{x-}^{\a}(a,x)$  and
 $$ \n f|_{(a,x)}\n_{H_{x-}^{\a}(a,x)} \le C_{\a,x}\n
 f\n_{H_{b-}^{\a}(a,b)}.$$
\item {\rm (Extension with left boundary condition)}\ If $f\in H_{a+}^{\a}(a,x)$
and we define $f^x(s) =
f(s)$ for $s\in (a,x)$ and $f^x(s)= 0$ otherwise, then
$f^x\in H_{a+}^{\a}(a,b)$ and
$$ \n f^x\n_{H_{a+}^{\a}(a,b)} \le C_{\a,x} \n f \n_{H_{a+}^{\a}(a,x)}.$$
\end{enumerate}
\end{lemma}

This lemma allows us to write, for $0<\a<\frac12$,
$$ H^\a (a,b) := H_{a+}^{\a}(a,b) =
H_{b-}^{\a}(a,b)$$ 
with equivalent norms. We shall use this simplified notation whenever the 
precise choice of the norm is irrelevant; when the choice of the norm does matter
we stick to the original notations with subscripts.

The next result is formulated for the right 
fractional integral; a similar result
holds for the left fractional integral.

\begin{lemma}\label{lem:indic}
Let $0<\a<\frac12$. For all $a\le x < y\le  b$, the indicator
function $1_{[x,y)}$ defines an element in $H^{\a}(a,b)$.
Moreover, the linear span of the indicator functions is dense in
$H^{\a}(a,b)$.
\end{lemma}

\begin{proof}
To prove the first assertion, by linearity we may assume that
$x=a$. Define
$$ g_y(t) := \frac1{\Gamma(1-\a)} (y-t)^{-\a} 1_{(a,y)}(t), \qquad t\in
(a,b).$$ Then $g_y\in L^2(a,b)$ and
$$ (I_{b-}^{\a} g_y)(t)
= \frac{1}{\Gamma(\a)\Gamma(1-\a)} \int_t^b (s-t)^{\a-1}  (y-s)^{-\a}
1_{(a,y)}(s)\,ds,
\qquad t\in (a,b).$$ For $t\in [y,b)$ it is clear that
$(I_{b-}^{\a} g_y)(t) = 0$, whereas for $t\in (a,y)$ we have, by
a change of variable $\sigma = (s-t)/(y-t)$,
$$
(I_{b-}^{\a} g_y)(t) = \frac{1}{\Gamma(\a)\Gamma(1-\a)} \int_0^1 \sigma^{\a-1}
(1-\sigma)^{-\a} \,d\sigma = 1.$$ This shows that
\begin{equation}\label{eq:indic}
 I_{b-}^{\a} g_y = 1_{(a,y)}
\end{equation}
and therefore $1_{(a,y)}\in H^{\a}(a,b).$

Next we prove that the linear span of all indicator functions of
the form $1_{[x,y)}$ is dense in $H^{\a}(a,b)$. Since
$I_{b-}^{\a}$ is an isomorphism from $L^2(a,b)$ onto
$H^\a(a,b)$, it is enough to prove that the linear span of
the set of functions $g_y$ introduced above is dense in
$L^2(a,b)$. To this end let us assume that $f\in L^2(a,b)$ is a
function such that for all $a< y\le b$ we have $ [f,
g_y]_{L^2(a,b)} = 0$. In view of
$$ 0 =[f, g_y]_{L^2(a,b)} = \frac1{\Gamma(1-\a)} \int_a^y (y-t)^{-\a} f(t)\,dt =
(I_{a+}^{1-\a} f)(y), \qquad  a< y\le b,$$ this implies that $I_{a+}^{1-\a} f =
0$ in
$H^{\a}(a,b)$. Therefore $f=0$ in $L^2(a,b)$ and the proof
is complete.
\end{proof}

For $\a>0$ we define the negative fractional integral spaces
$H_{0+}^{-\a}(a,b)$ and $H_{T-}^{-\a}(a,b)$ as the completions of $L^2(a,b)$
with respect to the norms
$$ \n f\n_{H_{a+}^{-\a}(a,b)} := \n I_{a+}^{\a} f\n_{L^2(a,b)}, \quad
 \n f\n_{H_{b-}^{-\a}(a,b)} := \n I_{b-}^{\a} f\n_{L^2(a,b)}.$$
It is an easy consequence of Lemma \ref{lem:restr} that for $0<\a<\frac12$
we have
$ H_{0+}^{-\a}(a,b) = H_{T-}^{-\a}(a,b)$
with equivalent norms. Accordingly we shall write
$$ H^{-\a}(a,b):= H_{0+}^{-\a}(a,b) = H_{T-}^{-\a}(a,b)$$
as long as the precise choice of the norm is unimportant.

We further define, for $\a>0$,
$$ I_{a+}^{-\a} := ( I_{a+}^{\a})^{-1} = D_{a+}^\a, \qquad
I_{b-}^{-\a} := ( I_{b-}^{\a})^{-1} = D_{b-}^\a,
$$
where $D_{a+}^\a$ and $D_{b-}^\a$ are the left- and right fractional
derivatives of order $\a$.
With these definitions, we have isometric isomorphisms
$$I_{a+}^{-\a}: L^2(a,b) \simeq H_{a+}^{-\a}(a,b) \qquad
I_{b-}^{-\a}: L^2(a,b) \simeq H_{b-}^{-\a}(a,b).$$

Finally we let $H_{a+}^0(a,b) = H_{b-}^0(a,b) :=  L^2(a,b)$ and agree
that $I_{a+}^0 = I_{b-}^0 := I$, the identity mapping on $L^2(a,b)$.

\section{Stochastic integration}\label{sec:stoch-int}

Throughout the rest of this paper we fix a number $T>0$. We shall write $L^2:=
L^2(0,T)$ and
$$ \bal
H_{0+}^{\a}&:= H_{0+}^{\a}(0,T), &&
& H_{T-}^{\a}&:= H_{T-}^{\a}(0,T),
 \\
C_{0+} &:=C_{0+}[0,T], &&
& C_{T-} &:=C_{T-}[0,T],
\eal
$$
where $C_{0+}[0,T] = \{f\in C[0,T]: \ f(0)=0\}$ and $C_{T-}[0,T] = \{f\in C[0,T]: \ f(T)=0\}$.
For $\a>\frac12$ we denote the inclusion mappings $\Ha \embed C_{0+}$ and $H_{T-}^{\a}\embed
C_{T-}$ by
$i_{0+}^{\a}$ and $i_{T-}^{\a}$, respectively.

For $0<\b<1$ and $0\le s,t\le T$ we set
$$ \Gamma_{s,t}^\b := [(i_{0+}^{\b+\frac12})\s
\d_s,(i_{0+}^{\b+\frac12})\s\d_t]_{ \Hb},$$
where $\d_s$ and $\d_t$ denote the Dirac measures concentrated at $s$ and $t$
(which we identify with functionals in the dual of $C_{0+}$ in the natural way).
An easy computation, cf. \cite[Section 6.2]{Fe-dLP}, gives
$$ \Gamma_{s,t}^\b
= \frac{1}{(\Gamma(\b+\frac12))^2} \int_0^{s\wedge t}
(s-u)^{\b-\frac12}(t-u)^{\b-\frac12}\,du, \qquad s,t\in [0,T].$$

\begin{definition}
A {\em Liouville fractional Brownian motion (Liouville fBm) of
order $0<\b<1$}, indexed by $[0,T]$, is a Gaussian process $\Wb =
(\Wb(t))_{t\in [0,T]}$ such that
$$ \E (\Wb(s) \Wb(t)) = \Gamma_{s,t}^{\b}, \qquad s,t\in[0,T].$$
\end{definition}

By the general theory of Gaussian processes, Liouville fBm exists. Note that 
$\Gamma_{s,t}^\frac12 = s\wedge t$ so a Liouville fBm of order $\frac12$ is just
a standard Brownian
motion.

The {\em stochastic integral} of a 
real-valued step function $f = \sum_{j=1}^N
c_j 1_{(a_j, b_j]}$ with respect to a Liouville fBm $\Wb$ is defined by
$$ \int_0^T f\,d\Wb := \sum_{j=1}^N c_j (\Wb(b_j) - \Wb(a_j)).$$
One easily checks that this definition does not depend on the
representation of $f$.
We proceed with analogues, for $0<\b<\frac12$ and $\frac12<\b<1$,
of the classical It\^o isometry (which corresponds to $\b=\frac12$).
These cases require different treatments and are therefore considered separately.

\subsection{The case $0<\b<\frac12$}

\begin{proposition}[It\^o isometry I]\label{prop:isom} Let $0<\b<\frac12$.
If $f:(0,T)\to\R$ is a step function, then
$\int_0^T f\,d\Wb$ is Gaussian and
\begin{equation}\label{eq:isom}
\E \Big| \int_0^T f\,d\Wb\Big|^2 = \n f\n_{\HbbT}^2.
\end{equation}
As a result, the mapping $ f\mapsto\int_0^T f\,d\Wb$ has a unique
extension to an isometry from $\HbbT$ into $L^2(\Om)$.
\end{proposition}
\begin{proof}
Suppose that $f =\sum_{j=1}^N c_j 1_{(a_j, b_j]}$ is a step
function 
with real coefficients $c_j$. 
We may assume that the intervals $(a_j,b_j]$ are disjoint.
Then,
$$
\begin{aligned}
  \E \Big| \int_0^T f\,d\Wb\Big|^2
 & = \E \Bigl|\sum_{j=1}^N c_j
(\Wb(b_j) - \Wb(a_j))\Big|^2
 \\ & =\sum_{i,j=1}^N  c_i c_j\, (\Gamma_{b_i,b_j}^{\b} - \Gamma_{a_i,b_j}^{\b}
- \Gamma_{a_j,b_i}^{\b}+\Gamma_{a_i,a_j}^{\b})
 = \int_0^T |g(s)|^2\,ds,
\end{aligned}
$$
where
$$ g(s) := \frac{1}{\Gb}\sum_{j=1}^N c_j
\bigl((b_j-s)^{\b-\frac12}\,1_{(0,b_j]}(s)
 -(a_j-s)^{\b-\frac12}\,1_{(0,a_j]}(s)\bigr).
$$
Since  $0<\b<\frac12$, \eqref{eq:indic} shows that
$$ \IbbT g = \sum_{j=1}^N c_j\, 1_{(a_j,b_j]} = f.$$
In view of the identity $\n g\n_{L^2} = \n f\n_{\HbbT}$, the
isometry \eqref{eq:isom} is proved.
The final assertion concerning the unique extendability of the
integral follows from the density of the step functions in $\HbbT$
as proved in Lemma \ref{lem:indic}.
\end{proof}

\subsection{The case $\frac12<\b<1$}

\begin{proposition}[It\^o isometry II]\label{prop:isom2} Let $\frac12<\b<1$.
If $f:(0,T)\to\R$ is a step function, then
$\int_0^T f\,d\Wb$ is Gaussian and
\begin{equation*}
\E \Big| \int_0^T f\,d\Wb\Big|^2 = \n f\n_{\HbbT}^2.
\end{equation*}
As a result, the mapping $ f\mapsto\int_0^T f\,d\Wb$ has a unique
extension to an isometry from $\HbbT$ into $L^2(\Om)$.
\end{proposition}
\begin{proof}
First let $f = 1_{(0,s)}$ and $g=1_{(0,t)}$ be left indicator functions with $s<t$. Then
\begin{align*}
\E \Big[ \int_0^T 1_{(0,s)}\,d\Wb \cdot \int_0^T 1_{(0,t)}\,d\Wb\Big]
& = \E [\Wb(s)\Wb(t)]
\\ & = \frac1{(\Gamma(\b+\frac12))^2}\int_0^s (s-u)^{\b-\frac12}(t-u)^{\b-\frac12}\,du.
\end{align*}
On the other hand, for $\tau\in\{s,t\}$,
$$
\bal
I_{T-}^{\b-\frac12} 1_{(0,\tau)}(u)
& =  \frac1{\Gamma(\b-\tfrac12)}\int_{u}^T (r-u)^{\b-\frac32}1_{(0,\tau)}(r)\,dr
\\ & = \frac1{\Gamma(\b-\tfrac12)}\int_{u\wedge \tau}^\tau (r-u)^{\b-\frac32}\,dr
\\ & = \frac1{\Gamma(\b+\tfrac12)}[(\tau-u)^{\b-\frac12}-((u\wedge
\tau)-u)^{\b-\frac12}].
\eal
$$
Hence,
$$
\bal
\ [ I_{T-}^{\b-\frac12} 1_{(0,s)},  I_{T-}^{\b-\frac12} 1_{(0,t)}]_{L^2}
& = \frac1{(\Gamma(\b+\tfrac12))^2}\int_0^T 
[(s-u)^{\b-\frac12}-((u\wedge s)-u)^{\b-\frac12}]
\\ & \qquad\qquad\qquad\qquad \cdot [(t-u)^{\b-\frac12}-((u\wedge t)-u)^{\b-\frac12}]\,du 
\\ & = \frac1{(\Gamma(\b+\tfrac12))^2}\int_0^{s\wedge t} (s-u)^{\b-\frac12}(t-u)^{\b-\frac12}\,du.
\eal
$$
Putting things together we obtain  $$\E\Big[\int_0^T 1_{(0,s)}\,d\Wb\cdot \int_0^T 1_{(0,t)}\,d\Wb\Big] =
\big[I_{T-}^{\b-\frac12} 1_{(0,s)}, I_{T-}^{\b-\frac12} 1_{(0,t)}\big]_{L^2}^2.$$
By linearity,
this identity extends to linear combinations of
left indicator functions. Therefore we obtain, for all step functions $\phi$,
$$\E \Big|\int_0^T \phi\,d\Wb\Big|^2 = \big\n I_{T-}^{\b-\frac12}
\phi\big\n_{L^2}^2.$$
Since step functions are dense in $L^2$, this proves the result.
\end{proof}

\subsection{}
We finish this section with a lemma that will be needed in Section
\ref{sec:ev-eq}.

\begin{lemma}\label{lem:estimate}
Let $W^\b$ be a Liouville fBm of order $0<\b<1$.
For all $0\le  \a < \min\{\b+\frac12,1\}$
and $0\le s<t<\infty$,
$$\Big(\E\Big|\int_s^t (t-r)^{-\a}\,dW^\b(r)\Big|^2\Big)^\frac12
= c_{\a,\b} (t-s)^{\b-\a},$$
where $c_{\a,\b}$ is a constant depending only on $\a$ and $\b$.
\end{lemma}
\begin{proof}
For $\b=\frac12$ the result is immediate from the classical It\^o isometry.

Next let $0<\b<\frac12$.
For $g_s(r) := (r-s)^{\b-\a-\frac12}$ we have
\begin{align*}
I^{\frac12-\b}_{s+} g_s(u)
 = \frac1{\Gamma(\frac12-\b)}\int_s^u     
(u-r)^{-\frac12-\b}(r-s)^{\b-\a-\frac12}\,dr
 = C_{\a,\b}(u-s)^{-\a},
\end{align*}
where $C_{\a,\b}$ is a constant depending only on $\a$ and $\b$. 
Hence by Lemma \ref{lem:properties}(3) and Proposition \ref{prop:isom},
\begin{align*}
\Big(\E\Big|\int_s^t (t-r)^{-\a}\,dW^\b(r)\Big|^2\Big)^\frac12
& = \n r\mapsto (t-r)^{-\a}\n_{H_{t-}^{\frac12-\b}}
  = \n r\mapsto (r-s)^{-\a}\n_{H_{s+}^{\frac12-\b}}
\\ & = \frac1{C_{\a,\b}}\n u\mapsto (u-s)^{\b-\a-\frac12}\n_{L^2(s,t)}
\\ & = c_{\a,\b}(t-s)^{\b-\a}.
\end{align*}

Next let $\frac12<\b<1$. By Lemma \ref{lem:properties}(3) and Proposition
\ref{prop:isom2},
\begin{align*}
\Big(\E\Big|\int_s^t (t-r)^{-\a}\,dW^\b(r)\Big|^2\Big)^\frac12
& = \n r\mapsto (t-r)^{-\a}\n_{H_{t-}^{\frac12-\b}}
  = \n r\mapsto (r-s)^{-\a}\n_{H_{s+}^{\frac12-\b}}
\\ &  = \Big\n u\mapsto \frac1{\Gamma(\b-\frac12)}\int_s^u
(u-r)^{\b-\frac32}(r-s)^{-\a}\,dr
\Big\n_{L^2(s,t)}
\\ &   = C_{\a,\b}'\Big\n u\mapsto
(u-s)^{\b-\a-\frac12}\Big\n_{L^2(s,t)}
\\ & = c_{a,b}'(t-s)^{\b-\a},
\end{align*}
where $C_{\a,\b}'$ and $c_{\a,\b}'$ are a constant depending only on
$\a$ and $\b$.
\end{proof}

\subsection{$\g$-Radonifying operators}

In order to prepare for the results on vector-valued stochastic integration
we need a couple of preliminaries on spaces of $\g$-radonifying operators.
For the rest of this paper we fix a real Hilbert space $H$ and a real Banach space $E$.
Unless otherwise stated, $[\cdot,\cdot]_H$ and $\n \cdot\n_H$ refer to the inner product and norm
of $H$, and $\n \cdot\n$ refers to the norm of $E$.

Any finite rank operator $S:\H\to E$ can be represented in the form
$$S = \sum_{n=1}^N h_n\otimes x_n$$ with  $h_1,\dots,h_N$ orthonormal in $\H$ and
$x_1,\dots,x_N$ taken from $E$.
The {\em $\g$-radonifying norm} of $S$ is then defined by
$$ \Big\n \sum_{n=1}^N h_n\otimes x_n \Big\n_{\g(\H,E)}^2 :=
\E \Big\n  \sum_{n=1}^N \g_n x_n \Big\n^2,$$
where $(\g_n)_{n\ge 1}$ is a sequence of independent standard Gaussian random
variables on some probability space $(\Om,\mathscr{A},\P)$.
It is easy to check that this definition does not depend on the particular
representation of $S$.
The completion of the space of finite rank operators with respect to
this norm is denoted by $\g(\H,E)$. This space is continuously embedded in
$\calL(\H,E)$, and a bounded operator $R\in \calL(\H,E)$ is called {\em
$\g$-radonifying} if it belongs to $\g(\H,E)$. 

The space
$\g(\H,E)$ is an {\em operator ideal} in $\calL(\H,E)$ in the sense that 
whenever $H'$ is another real Hilbert space, $E'$ is another real Banach space, and 
$R:\H'\to \H$ and $T:E\to E'$
are bounded operators, then $S\in \g(\H,E)$ implies
$TSR\in \g(\H',E')$ and
\begin{equation*}
 \n TSR\n_{ \g(\H',E')} \leq \n T\n \n S\n_{\g(\H,E)} \n R\n.
\end{equation*}

If $E$ is a Hilbert space, then $\g(\H,E)$ is isometrically
isomorphic to the Hilbert space of
Hilbert--Schmidt operators from $\H$ to $E$.

\begin{example}[\cite{Brz-Nee, NVW}]
For $E = L^q(X,\mu)$ with $q\in [1,\infty)$ and $(X,\mu)$ a $\sigma$-finite
measure space we have a natural isomorphism of Banach spaces
$$
\g(\H,L^q(X,\mu))\simeq L^q(X,\mu;\H)
$$
obtained by assigning to a function $f\in L^p(X,\mu;\H)$ the operator
$S_f: \H\to L^p(X,\mu)$, $S_f h:= [f(\cdot),h]$.
\end{example}

\begin{lemma}\label{lem:tensor} Let $S = \sum_{n=1}^N h_n\otimes x_n$ be a
finite rank operator from $\H$ to $E$, with  $h_1,\dots,h_N$ orthonormal in $\H$
and $x_1,\dots,x_N$ taken from $E$, and
suppose $\H'$ is another Hilbert space. For all $h'\in \H'$ we have
$$ \Big\n \sum_{n=1}^N (h'\otimes h_n)\otimes x_n\Big\n_{\g(\H'\widehat\otimes
\H,E)} = \n h'\n \Big\n  \sum_{n=1}^N h_n\otimes x_n\Big\n_{\g(\H,E)}.$$
Here $\H'\widehat\otimes \H$ denotes the Hilbert space tensor product of $\H'$ and
$\H$.
\end{lemma}

For the proof, just note that if $\n h'\n$ is normalised to $1$, then the
vectors $h'\otimes h_n$ are orthonormal in $\H'\widehat\otimes \H$.

A bounded operator $T\in\calL(\H,E)$ is said to be {\em $\gamma$-summing}
if 
$$\n T\n_{\g_\infty(\H,E)}^2 := \sup_h \, \E \Big\n \sum_{n=1}^N \g_n Th_n\Big\n^2 <\infty,
$$ the supremum being taken over all finite orthonormal
systems $h = (h_n)_{n=1}^N$ in $\H$. 
Endowed with the above norm, the space $\gamma_\infty(H,E)$ of all $\gamma$-summing operators from $H$ to $E$ 
is a Banach space.
Every $\g$-radonifying operator is
$\gamma$-summing, and the inclusion $\g(\H,E)\subseteq \g_\infty(\H,E)$ is
isometric. It follows from a theorem of Hoffmann-J{\o}rgensen and Kwa\-pie\'n
that equality $\g(\H,E)= \g_\infty(\H,E)$ holds when $E$ does not contain a closed
subspace isomorphic to $c_0$. For proofs and more information we refer
to the survey paper \cite{Nee-Canberra} and the reference given therein.

Let $\Phi:\R_+\to \calL(\H,E)$ be an {\em $\H$-strongly measurable function}, i.e.
$\Phi h$ is strongly measurable for all $h\in \H$, and suppose that $\Phi\s x\s
\in L^2(\R_+;\H)$
for all $x\s\in E\s$. We say that an operator $R\in \calL(L^2(\R_+;\H),E)$
 is {\em represented} by $\Phi$ if we have
\begin{equation*} R^* x\s = \Phi\s x\s \end{equation*}
in $L^2(\R_+;\H)$ for all $x\s\in E\s$.

A family $\mathscr{T}$ of bounded linear operators from a Banach space $E$ to
another Banach space $F$
is called {\em $\g$-bounded} if there exists a finite constant $C$ such
that for all finite sequences $(x_n)_{n=1}^N$ in $E$ and
$(T_n)_{n=1}^N$ in ${\mathscr {T}}$ we have
\[ \E \Big\n \sum_{n=1}^N \g_n T_n x_n\Big\n^2
\le C^2\E \Big\n \sum_{n=1}^N \g_n x_n\Big\n^2.
\]
The least admissible constant $C$ is called the {\em $\g$-bound} of
$\mathscr {T}$, notation $\g(\mathscr{T})$. An important way of generating
$\g$-bounded families is due to Weis \cite{Wei} who showed that if
$f: (0,T)\to \calL(E,F)$ is continuously differentiable with integrable
derivative, then $\mathscr{T}_f = \{f(t): f\in (0,T)\}$ is $\g$-bounded and
\beq\label{eq:weis} \g(\mathscr{T}_f) \le \n f(0+)\n + \int_0^T \n
f'(t)\n\,dt.\eeq
As application of this result is contained in Lemma \ref{lem:g-bdd} below.

We continue with a multiplier result of Kalton and Weis \cite{KW}
(see \cite{Nee-Canberra} for a proof)
which connects the notions of radonification and $\g$-boundedness.

\begin{lemma}\label{lem:multiplier}
Let $M:(0,T)\to \calL(E,F)$ be a function with the following
properties:
\begin{enumerate}
\item for all $x\in E$ the function $Mx$ is strongly measurable;

\item the range $\mathscr{M} = \{M(t): \ t\in(0,T)\}$ is $\gamma$-bounded.
\end{enumerate}
Then for all functions $\Phi:(0,T)\to \calL(\H,E)$ representing an operator
$S_\Phi\in \g(L^2(0,T;\H),E)$, the function
$M\Phi: (0,T)\to \calL(\H,F)$ represents an operator $S_{M\Phi}\in
\g_\infty(L^2(0,T;\H),F)$ and
\begin{equation*}
\|S_{M\Phi}\|_{\g_\infty(L^2(0,T;\H),F)}\le
\gamma(\mathscr{M})\|S_{\Phi}\|_{\g(L^2(0,T;\H),E)}.
\end{equation*}
\end{lemma}

In many situations (such as in the application of this lemma in Section
\ref{sec:ev-eq}) 
one actually has $S_{M\Phi}\in \g(L^2(0,T;\H),F)$, for instance by 
an application of Theorem \ref{thm:suff1}. 
In view of the isometric inclusion $\g(L^2(0,T;\H),F)\subseteq \g_\infty(L^2(0,T;\H),F)$,
the estimate of 
Lemma \ref{lem:multiplier} then takes the form
$$ \|S_{M\Phi}\|_{\g(L^2(0,T;\H),F)}\le
\gamma(\mathscr{M})\|S_{\Phi}\|_{\g(L^2(0,T;\H),E)}.
$$ 

\subsection{Stochastic integration in Banach spaces}

Let $\HH$ be a Hilbert space and $(\Om,\F,\P)$ a probability space. A mapping $W:
\HH\to L^2(\Om)$
is called an {\em $\HH$-isonormal process} if
$W(h)$ is centred Gaussian for all $h\in\HH$ and
$$\E W(h_1)W(h_2) = [h_1,h_2]_{\HH}, \quad h_1,h_2\in \HH.$$
By Proposition \ref{prop:isom} (for $0<\b<\frac12$), Proposition
\ref{prop:isom2} (for $\frac12<\b<1)$ and the classical It\^o isometry (for
$\b=\frac12$), for all $0<\b<1$ the mapping
$$W^\b: f\mapsto \int_0^T f\,dW^\b,$$ initially defined for step functions $f$,
has a unique extension to an $\HbbT$-isonormal process. This observation
suggests the following definition.

\begin{definition}
Let $\H$ be a Hilbert space and let $0<\b<1$. An {\em $\H$-cylindrical Liouville
fBm of order $\b$}, indexed by $[0,T]$, is an $\HbbTH$-isonormal process.
\end{definition}

Here the Hilbert space $$\HbbTH:= \HbbT(0,T;\H)$$ is defined in the obvious way
using the right
fractional integral operators acting in $L^2(\H):= L^2(0,T;\H)$. It
is easy to see that $\HbbTH$ can be identified isometrically with
the Hilbert space completion of the tensor product $\HbbT\otimes \H$.

Our next task is to define an integral for $E$-valued function with respect to a
Liouville fBm, and more generally for $\calL(\H,E)$-valued
functions with respect to an $\H$-cylindrical Liouville fBm $W_\H^\b$, where
$\H$ is a real Hilbert space.
We shall proceed directly with the latter, as the former corresponds to the
special case $\H=\R$.  We follow \cite{Nee-Canberra}, which puts the approach of
\cite{NW} into an abstract format.

For an {\em elementary rank one} function $\Phi:(0,T)\to \calL(\H,E)$, i.e. a
function of the form
$$ \Phi = f \otimes (h\otimes x),$$
where $f\in \HbbT$ and $h\otimes x \in \calL(\H,E)$ is the rank one operator
$h'\mapsto [h',h]x$,
we define
$$ \int_0^T \Phi \,d\WbH := \WbH(f\otimes h)\otimes x.$$
This definition is extended by linearity to all {\em finite rank elementary
functions}
$\Phi:(0,T)\to\calL(\H,E)$, i.e. linear combinations of elementary rank one
functions. Any such function $$\Phi = \sum_{n=1}^N f_n \otimes (h_n\otimes
x_n)$$
defines a finite rank operator $R_\Phi:\HbbTH\to E$
by
$$ R_\Phi := \sum_{n=1}^N (f_n \otimes h_n)\otimes x_n.$$
It is immediate to verify that for all $x\s\in E\s$ we have
\beq\label{eq:Phis}
R_\Phi\s x\s = \Phi\s x\s = \sum_{n=1}^N \lb x_n,x\s\rb(f_n \otimes h_n)
\eeq
as elements of $\HbbTH$. Applying the results of \cite{Nee-Canberra, NW} to the
Hilbert space $\HbbTH$ we obtain:

\begin{theorem}[It\^o isometry]\label{thm:SI} Let $\WbH$ be a cylindrical
Liouville fBm of order $0<\b<1$.
For all elementary finite rank functions
$\Phi:(0,T)\to \calL(\H,E)$ we have
$$ \E \Big\n \int_0^T \Phi \,d\WbH \Big\n^2 = \n R_\Phi\n_{\g(\HbbTH,E)}^2.$$
\end{theorem}

As a result, the $E$-valued stochastic integral with respect to $\WbH$ has a
unique extension to an isometry from $\g(\HbbTH,E)$ into $L^2(\Om;E)$.

Motivated by \eqref{eq:Phis} we shall call a function $\Phi:(0,T)\to \calL(\H,E)$
{\em stochastically integrable} with respect to $W_\H^\b$ if $\Phi\s x\s\in
\HbbTH$ for all $x\s\in E\s$ and there exists an operator $R\in \g(\HbbTH,E)$
such that
$$ R\s x\s = \Phi\s x\s$$
in $\HbbTH$ for all $x\s\in E\s$.
The operator $R$, if it exists, is uniquely determined. In this situation we say
that $\Phi$ {\em represents} $R$.

Using the right ideal property for spaces of $\g$-radonifying operators, applied
to the embeddings $\HbbTH\embed L^2(0,T;\H)$ (for $0<\b<\frac12$)
and $L^2(0,T;\H)\embed \HbbTH$ (for $\frac12<\b<1)$ we obtain the
first part of the following
simple consequence of Theorem \ref{thm:SI}; the second part is proved similarly.

\begin{corollary}\label{cor:SI}  Let $\WbH$ be an $\H$-cylindrical Liouville fBm,
$W_\H$ an $\H$-cylindrical Brownian motion,
and consider a function $\Phi:(0,T)\to\calL(\H,E)$.

\ben
\item If $0<\b<\frac12$ and $\Phi$ is stochastically integrable with respect to
$W_\H^\b$, then
$\Phi$ is stochastically integrable with respect to $W_\H$ as well.
\item If $\frac12<\b<1$ and $\Phi$ is stochastically integrable with respect to
$W_\H$, then
$\Phi$ is stochastically integrable with respect to $\WbH$ as well.
\een
In fact, for any two numbers $0<\beta_1 <\beta_2 <1$, stochastic integrability with respect to
$W_\H^{\b_1}$ implies stochastic integrability with respect to
$W_\H^{\b_2}$.
\end{corollary}

We proceed with two further sufficient conditions for stochastic integrability.
In both we assume that $\WbH$ is an $\H$-cylindrical Liouville fBm of order $\b$
and $W_\H$ is an $\H$-cylindrical Brownian motion.

The first theorem is a simple adaptation of a result due to Kalton and Weis for
$\b=\frac12$ and $\H=\R$ \cite{KW}.

\begin{theorem}\label{thm:suff1} Let $0<\b<\frac12$.
If $\Phi:(0,T)\to \g(\H,E)$ is a continuously differentiable function that
satisfies
$$ \int_0^T t^{\b} \n \Phi'(t)\n_{\g(\H,E)}\,dt<\infty,$$
then $\Phi$ is stochastically integrable with respect to $W_\H^\b$ and we have
$$\Big(\E \Big\n \int_0^T \Phi\,dW_\H^\b \Big\n^2\Big)^\frac12
\le C_\b T^{\b} \n \Phi(T-)\n_{\g(\H,E)} + C_\b \int_0^T t^{\b} \n
\Phi'(t)\n_{\g(\H,E)}\,dt,
$$
where $C_\b = 1/{\sqrt{2\b}\,\Gamma(\frac12+\b)}$.
\end{theorem}
\begin{proof}
Put $g(s,t) := \one_{(t,T)}(s)f'(s)$ for $s,t \in (0,T)$. Then,
$$f(t) = f(T-) - \int_0^T g(s,t)\, ds$$ for all $t\in (0,T)$.
Using Lemma \ref{lem:indic}, for almost all $s\in (0,T)$ the function $t\mapsto
g(s,t) = \one_{(t,T)}(s)f'(s) = \one_{(0,s)}(t)f'(s)$ belongs to $\gamma
(\HbbTH,E)$ with norm
\[
\bal
\n \one_{(\cdot,T)}(s)f'(s) \n_{\gamma (\HbbTH,E)}
& = \n \one_{(0,s)}\n_{\HbbT}\n f'(s)\n_{\gamma(\H,E)}
\\ & = C_\b s^\b \n f'(s)\n_{\gamma(\H,E)}.
\eal
\]
It follows that the $\gamma (\HbbTH,E)$-valued function $s \mapsto g(s, \cdot )$
is
Bochner integrable. Identifying the operator $f(T-)\in \gamma(\H,E)$ with the
constant function $\one_{(0,T)}f(T-)\in \gamma(\HbbTH,E)$, we find that
$f \in \gamma (\HbbTH,E)$ and
\begin{align*}
\n f\n_{\gamma (\HbbTH,E)} &  \leq
C_\b T^\b\n f(T-)\n_{\g(\H,E)} + \int_0^T\n g(s,\cdot )\n_{\gamma (\HbbTH,E)}\,
ds\\
& = C_\b T^\b\n f(T-)\n_{\g(\H,E)} + C_\b \int_0^T s^\b\n f'(s)\n_{\g(\H,E)}\, ds.
\end{align*}
\end{proof}

The second theorem gives an improvement to Corollary \ref{cor:SI}(2).

\begin{theorem}\label{thm:suff2} Let $\frac12<\b<1$ and $0\le \a<\b-\frac12$.
If $\Phi:(0,T)\to \calL(\H,E)$ is stochastically integrable with respect to $W_\H$,
then $$t\mapsto t^{\a}\Phi(t)$$ is stochastically integrable with
respect to $W_\H^\b$.
\end{theorem}
\begin{proof}
This is proved in the same way as Corollary \ref{cor:SI}(2), except that now we
apply the right ideal property of $\g$-radonifying operators, now applied to the
bounded operator $K_{\a,\b}$ on $L^2(0,T;\H)$, $$K_{\a,\b}f(t) :=
t^{-\a}I_{0+}^{\b-\frac12}f(t).$$
\end{proof}

\section{Comparison with classical fBm}

In this section we compare Liouville fBm $W^\b$ with classical fBm, that is,
a Gaussian process $(\wt W^\b(t))_{t\in [0,T]}$ with covariance
$$\E \wt W^\b(s)\wt W^\b(t) = s^{2\b} + t^{2\b} - |t-s|^{2\b},$$
where $\beta \in (0,1)$ is the so-called {\em Hurst parameter}.
Brownian motion again corresponds to the case $\b=\frac12$. For a review of the
theory of stochastic integration with respect to classical fBm we refer to 
\cite{book, Nua}.

Let $\HH^\b$ and $\wt \HH^\b$ be the Hilbert spaces obtained as the completions of
the step functions
with respect to the scalar products
$$ [1_{(0,s)}, 1_{(0,t)}]_{\HH^\b} := \E W^\b(s)W^\b(t)$$
and
$$ [1_{(0,s)}, 1_{(0,t)}]_{\wt \HH^\b} := \E \wt W^\b(s)\wt W^\b(t)$$
for Liouville fBm and classical fBm, respectively. 

\begin{proposition}\label{prop:compare} For all $0<\b<\frac12$ we have
$\HH^{\b} = \wt \HH^\b  =  H^{\frac12-\b}_{T-}$ with equivalent norms.
\end{proposition}
\begin{proof}
We have already seen that $\HH^{\b} =H^{\frac12-\b}_{T-}$  isometrically.
The fact that $\wt \HH^\b = H^{\frac12-\b}_{T-}$ up to an equivalent norm is
well known; see \cite[Proposition 8]{Alos}, \cite[Formulas (2.27)]{book}, and
\cite{DecrUst}.
\end{proof}

By the very definition of an isonormal process we have the It\^o isometry
$$
\E \Big| \int_0^T f\,d\WbH\Big|^2 = \n f\n_{\HH^\b}^2.
$$
Similarly, it is well known \cite{book, Nua} that
$$
\E \Big| \int_0^T f\,d\wt W^\b_\H\Big|^2 = \n f\n_{\wt \HH^\b}^2,
$$
where $\wt W^\b_\H$ is $H$-cylindrical classical fBm. 
Having observed the latter, we can repeat the constructions of the previous
section and obtain analogues of our results for classical fBm with Hurst parameter
$0<\beta<\frac12$. The following result relates the two stochastic integrals.

\begin{theorem}\label{thm:equiv} Let $0<\b<\frac12$.
For a function $\Phi:(0,T)\to\calL(\H,E)$ the following are equivalent:
\ben
\item[\rm(1)] $\Phi$ is stochastically
integrable with respect to $\wt W_\H^\b$;
\item[\rm(2)] $\Phi$ is stochastically
integrable with respect to $\WbH$.
\een
In this situation we have
$$ \E\Big\n \int_0^T \Phi\,d\wt W_\H^\b\Big\n^2 \eqsim
\E\Big\n \int_0^T \Phi\,d\WbH\Big\n^2
$$ with two-sided constants independent of $\Phi$.
\end{theorem}
\begin{proof}
In view of Proposition \ref{prop:compare}, the first assertion is immediate from
Theorem \ref{thm:SI} and its counterpart for classical fBm (which again holds by
the abstract results of \cite{Nee-Canberra}, now applied to the Hilbert space
$\wt\H^\b$).
To prove the equivalence of norms, we first observe that for all $x\s\in E\s$,
$$ \E\Big| \int_0^T \Phi\s x\s \,d\WbH\Big|^2
= \n \Phi\s x\s\n_{\wt\H^{\b}}^2
$$
and similarly
$$ \E\Big| \int_0^T \Phi\s x\s \,d\wt W_\H^{\b}\Big|^2
 = \n \Phi\s x\s\n_{\wt\H^\b}^2.
$$
Hence by Proposition \ref{prop:compare},
$$ \E\Big| \int_0^T \Phi\s x\s \,d\WbH\Big|^2 \eqsim
\E\Big| \int_0^T \Phi\s x\s \,d\wt W_\H^{\b}\Big|^2
$$
with two-sided constants independent of $\Phi$ and $x\s$.
The result now follows from a standard comparison result for Banach space-valued
Gaussian random variables.
\end{proof}

\begin{remark}
For $\frac12<\b<1$, the spaces $\HH^\b$ and  $\wt \HH^\b$ are different:
the former consists of
consists of all distributions
$\psi$ such that $I_{T-}^{\b-\frac12}\psi \in L^2,$ whereas the latter
consists of those distributions for which
$$s\mapsto
s^{\frac12-\b}
(I_{T-}^{\b-\frac12} (u\mapsto u^{\b-\frac12}\psi(u)))(s) \in L^2.$$ See
\cite[Formula 2.18]{book}.
\end{remark}

\section{Evolution equations driven by Liouville fBm}\label{sec:ev-eq}

In this section we shall apply the results of the previous
sections to study the stochastic abstract Cauchy problem
\begin{equation}\label{sACP}\tag{sACP}
\left\{
\begin{aligned}
dU(t) & = AU(t) \,dt + B\,dW_\H^\b(t), \qquad t\in [0,T],\\
 U(0) & = x.
\end{aligned}
\right.
\end{equation}
Here $A$ is the generator of a $C_0$-semigroup $S=\{S(t)\}_{t\ge 0}$ on $E$,
$B\in\calL(\H,E)$ is a given bounded linear operator,
and $W_\H^\b$ is a Liouville cylindrical fBm of order $0<\b<1$ on a probability
space $(\Om,\F,\P)$.

If, for all $t>0$, the $\calL(\H,E)$-valued function $S(t-\cdot)B$ is
stochastically integrable on $(0,t)$ with respect to $W_\H^\b$, the process
\begin{equation}\label{eq:repr-sol}
U^x(t) = S(t) x + \int^t_0 S(t-s)B\,dW_\H^\b(s)
\end{equation}
is called the {\em mild solution} of \eqref{sACP}.
It is an easy consequence of the definition of the stochastic integral that the process $U^x$
is strongly measurable as a mapping from $[0,\infty)\times \Omega$ into $E$.

The next theorem asserts the existence of a mild solution in the case 
where the Banach space $E$ is of 
type $2$  and the operator $B$ is $\g$-radonifying. To some extent 
this could be seen as a generalization of some results from \cite{Brz, NW}; 
see also \cite{Brz+Zab_2010} where a case of equations driven by a non-Gaussian 
L{\'e}vy process is considered. 

\begin{theorem}\label{thm:sACP1} 
Let $S$ be a $C_0$-semigroup on a Banach space $E$ with type $p\in (1,2]$. Then for all
$\b\in (\frac1p,1)$ and $B\in\g(\H,E)$ the function $S(t-\cdot)B$  is stochastically integrable on
$(0,t)$ with respect to $W_\H^\b$ for all $t>0$. As a consequence, the problem
\eqref{sACP}
has a unique mild solution $U$ which given by \eqref{eq:repr-sol}.
\end{theorem}

\begin{proof}
First assume that $E$ has type $2$. In that case, 
we have a continuous embedding
$$L^2(0,T; \gamma(\H,E))\embed 
\gamma(L^2(0,T;\H),E).$$ Evidently, $S(\cdot)B$
belongs to $L^2(0,T;\gamma(\H,E))$, and therefore 
this function is stochastically integrable
with respect to $\H$-cylindrical Brownian motions $W_\H$. The 
result then follows from Corollary \ref{cor:SI}.

Next assume that $1<p<2$.
By the results of 
\cite{KNVW, NVW-conditions}, for Banach space $E$ with type $p$ we have 
a continuous embedding 
$$ B_{p,p}^{\frac1p-\frac12}(0,T; \gamma(\H,E))\embed 
\gamma(L^2(0,T;\H),E).$$
By \cite[Theorem 4.6.1]{Triebel}
one has continuous embeddings 
$$H^{\b-\frac12}(0,T; \gamma(\H,E))\embed
B_{2,2}^{\b-\frac12}(0,T; \gamma(\H,E))\embed
B_{p,p}^{\frac1p-\frac12}(0,T; \gamma(\H,E)) 
$$
Combining these, we obtain a continuous embedding
$$H^{\b-\frac12}(0,T; \gamma(\H,E)) \embed \gamma(L^2(0,T;\H),E).$$
Recalling that this embedding is given, for finite element rank functions, by
$$f\otimes (h\otimes x) \mapsto (f\otimes h)\otimes x,$$ 
the isometry $I_{T-}^
{\beta-\frac12}: L^2(0,T) \mapsto H_{T-}^{\beta-\frac12}$ induces continuous embedding
\begin{equation}\label{eq:Hembedding}
L^2(0,T; \gamma(\H,E))\embed 
\gamma(H^{\frac12-\b}(0,T;\H),E).
\end{equation}
Now we may apply Theorem \ref{thm:SI}.
\end{proof}

This result is sharp in the following sense. 

\begin{example} Suppose $\frac12 < \beta<1$ is 
given. Then for any $1\le p < \frac1\beta$ 
there exists a Banach space $E$ with type $p$, a $C_0$-semigroup $S$
with generator $A$ on $E$, and a vector $x\in E$ for which the problem
\begin{equation}\label{eq:counter}
\left\{\begin{aligned} 
dU(t) & = AU(t)\,dt + x \,dw^\beta(t), \quad t\in [0,T],\\ 
U(0) & = 0,
\end{aligned}
\right.
\end{equation} 
fails to have a mild solution. Here, $w^\b$ is a real-valued Liouville fBm
with Hurst parameter $\b$. 
Note that \eqref{eq:counter} corresponds to the special case of \eqref{sACP}
for $\H = \R$ (identifying $x\in E$ with $1\otimes x \in\gamma(\R;E)$).

Indeed, let $1\le p < \frac1\beta$. On $E := L^p(0,T)$, let $S$ denote the left translation semigroup
on $E$,
$$
S(t)x(s) = \left\{
\begin{array}{cl}
x(s+t), & s+t<T,\\
0, & \hbox{otherwise}.
\end{array}
\right.$$ By combining
 Theorem \ref{thm:SI} and \cite[Theorem 2.3]{Brz-Nee} (see also \cite[Lemma 2.1]{NVW2}), 
for a given $x\in L^p(0,T)$ the problem \eqref{eq:counter} has a weak solution if only if $S(\cdot)x$ defines 
an element of $\gamma(H^{\frac12-\beta},L^p(0,T)) \simeq L^p(0,T;H^{\frac12-\beta})$
 (we need not
worry about boundary conditions since $0<\beta-\frac12<\frac12$; see 
the remarks at the end of Section \ref{sec:frac}).

Let us now suppose that this is true for all $x\in L^p(0,T)$.
Fix a number $0<\delta<T$ and consider an arbitrary function $x\in L^p(0,T)$ with support in $(\delta,T)$.
For almost all $t\in (0,T)$ it follows that
$s\mapsto S(t)x(s) = \one_{\{s+t<T\}} x(s+t)$ belongs to  $H^{\frac12-\beta}$.
In particular it follows that $s\mapsto x(s+t)$ belongs to $H^{\frac12-\beta}$
for almost all $t\in (0,\delta)$. This function being identically zero on the intervals
$(0,\delta-t)$ and $(T-t,T)$, it is immediate to see that its restriction to
$(\delta-t,T-t)$ belongs to $H^{\frac12-\beta}(\delta-t,T-t)$
This implies that $s\mapsto x(s)$ is in $H^{\frac12-\beta}(\delta,T)$.

By the closed graph theorem, this proves that we have a continuous inclusion
$ L^p(\delta,T) \embed H^{\frac12-\beta}(\delta,T).$
This is the same as saying that the fractional integral operator 
$I_{T-}^{\beta-\frac12}$ acts boundedly from $L^p(\delta,T)$ to $L^2(\delta,T)$.
The latter is known to false if $\beta <  \frac1p$. 

Hence there must exist $x\in L^p(0,T)$ for which the problem  \eqref{eq:counter} has no mild solution.
\end{example}
 
\begin{remark} By a result of Veraar (personal communication), for $p$-concave Banach lattices $E$
(such spaces have type $p$) one has a continuous embedding
$$
L^p(0,T; \gamma(\H,E))\embed 
\gamma(H^{\frac12-\frac1p}(0,T;\H),E).
$$
and therefore \eqref{eq:Hembedding} holds with
$\b=\frac1p$. As a consequence, for such spaces $E$, Theorem \ref{thm:sACP1} also 
holds for the critical exponent $\beta = \frac1p$.
We do not know whether this extends to arbitrary Banach spaces with type $p$.
\end{remark}

We return to the setting where $E$ is an arbitrary real Banach space.
In the proof of the next theorem we will need the following result, which is a
direct consequence of \eqref{eq:weis} combined with standard estimates for
analytic semigroups (cf. \cite{Paz}):

\begin{lemma}\label{lem:g-bdd}
Let $A$ generate an analytic $C_0$-semigroup $E$. Then for all $0\le
\theta<\eta$ and large enough $w\in\R$  the
set $$\{ t^{\eta}(w-A)^\theta S(t)\}: \ t\in (0,T)\}$$
is $\g$-bounded.
\end{lemma}

The main result of this section is an extension of a result of \cite{DMP2},
where it was assumed that $0<\b<\frac12$ and that $E$ is a Hilbert space.

\begin{theorem}\label{thm:sACP2} Let $0<\b<1$.
If $S$ is an analytic $C_0$-semigroup on an arbitrary Banach space $E$, then for all
$B\in\g(\H,E)$ the function $S(t-\cdot)B$  is stochastically integrable on
$(0,t)$ with respect to $W_\H^\b$.  As a consequence, the problem \eqref{sACP}
has a mild solution $U$ given by \eqref{eq:repr-sol}.
Moreover, for all $1\le p<\infty$ and all $\a,\theta\ge 0$ satisfying
$\a+\theta<\b$
we have $$U\in L^p(\Om;C^\a([0,T];E_\theta)),$$
where $E_\theta$ denotes the fractional domain space
of exponent $\theta$ associated with $A$.
\end{theorem}

\begin{remark}
If $0<\b<\frac12$, Theorem \ref{thm:equiv} enables us to replace the cylindrical
Liouville fBm by a cylindrical fBm.
\end{remark}

\begin{proof}
For $\frac12\le \b<1$, the existence of a mild solution follows from the fact
that
$S(t\mapsto\cdot)B$ is stochastically integrable on $(0,t)$ with respect to
every
$\H$-cylindrical Brownian motion $W_\H$ and Corollary \ref{cor:SI}.

For $0<\b<\frac12$ we verify the condition stated in Theorem \ref{thm:suff1}.
With $\Phi(t) = S(t)B$ we have
$$ t^{\b}\n \Phi'(t)\n_{\g(\H,E)} =
 t^{\b}\n AS(t)B\n_{\g(\H,E)}
 \le C t^{-1+\b}\n B\n_{\g(\H,E)},
$$
where $C$ is a constant depending only on $T$ and the semigroup $S$.
Since the function on the right-hand side is integrable the result follows from
Theorem \ref{thm:suff1}.

Next we prove the space-time regularity assertion.
We follow the proof of \cite[Theorem 10.19]{ISEM};
for the reader's convenience we include the details.
By the Kahane-Khintchine inequality we may assume that $p=2$.

Fix $\theta\ge 0$ and $\a\ge 0$ such that $\a+\theta<\b$.
We claim that for all $t\in [0,T]$ the random $U(t)$ takes
its values in $E_\theta$ almost surely. We prove this by showing that the
functions $S(t-\cdot)B$
are stochastically integrable on $(0,t)$ with respect to $W_\H^\b$ as an
$\calL(\H,E_\theta)$-valued function.
Indeed, this follows from Lemmas \ref{lem:g-bdd} and \ref{lem:multiplier},
once we realise three things:
\ben
\item[\rm(i)] For all $0<\eta<\theta$ the set $\{r^\eta  S(r): \ r\in (0,T)\}$
is
$\g$-bounded in $\calL(E,E_\theta)$ by Lemma \ref{lem:g-bdd};
\item[\rm(ii)] The function  $s\mapsto (t-s)^{-\eta} B$
represents an operator in $\g(H^{\frac12-\b}(0,t;\H),E)$ of norm $ \n
s\mapsto (t-s)^{-\eta}\n_{H^{\frac12-\b}(0,t)}
\n B\n_{\g(\H,E)}$
by Lemmas \ref{lem:estimate} and \ref{lem:tensor};
\item[\rm(iii)] $S(t-s)B = [(t-s)^\eta S(t-s)][(t-s)^{-\eta}B]$.
\een
Variations of this argument will be used repeatedly below.

Fix  $0\le s\le t\le T$. By the triangle inequality in $L^2(\Om;E)$,
$$
\bal
\big(\E \n U(t)-U(s)\n_{E_\theta}^2\big)^\frac12
 & \le \Big(\E \Big\n\int_0^s [S(t-r)-S(s-r)]B\,dW^\b(r)
\Big\n_{E_\theta}^2\Big)^\frac12
\\ & \qquad\qquad\quad\  + \Big(\E \Big\n \int_s^t S(t-r)B\,dW^\b(r)
\Big\n_{E_\theta}^2\Big)^\frac12.
\eal
$$
Choose $\la\in\R$ sufficiently large in order that the fractional powers of
$\la-A$
exist.
For the first term we have, for any choice of
$\e>0$ such that $\a+\theta+\e<\b$,
\begin{align*}
\ & \E \Big\n \int_0^s [S(t-r)-S(s-r)]B\,dW_\H^\b(r) \Big\n_{E_\theta}^2
\\ & \simeq \E \Big\n \int_0^s (s-r)^{\a+\theta+\e}(\la-A)^{\a+\theta} S(s-r)
\\ & \hskip3cm \times
(s-r)^{-\a-\theta-\e} [S(t-s)-I](\la-A)^{-\a}B\,dW_\H^{\b}(r) \Big\n^2
\\ & \stackrel{\rm(i)}{\le} C^2 \E \Big\n \int_0^s (s-r)^{-\a-\theta-\e}
[S(t-s)-I](\la-A)^{-\a}
B\,dW_\H^\b(r) \Big\n^2
\\ & \stackrel{\rm(ii)}{=} C^2 \n [S(t-s)-I](\la-A)^{-\a} B\n_{\g(\H,E)}^2
\E\Big|\int_0^s
(s-r)^{-\a-\theta-\e}\,dW^\b (r)\Big|^2
\\ & \stackrel{\rm(iii)}{=}  C^2 s^{2\b-2\a-2\theta-2\e} \n [S(t-s)-I](\la-A)^{-\a}\n^2 \n
B\n_{\g(\H,E)}^2
\\ & \stackrel{\rm(iv)}{\le} C^2T^2(t-s)^{2\a}\n B\n_{\g(\H,E)}^2,
\end{align*}
where the numerical value of $C$ changes from line to line.
In (i) we used Lemmas \ref{lem:multiplier} and \ref{lem:g-bdd} 
in combination with Theorem \ref{thm:SI}.
In (ii) we used Lemma \ref{lem:tensor} and Theorem \ref{thm:SI} in combination with an
approximation argument to see that if $W^\beta$
is any real-valued Liouville fBm, then for all 
$f\in \HbbT$ and $\tilde B\in \gamma(H,E)$,
$$
 \E \Big\n \int_0^T f(t)\tilde B\,dW_\H^\b(t) \Big\n^2 = \n \tilde B\n_{\g(\H,E)}^2
\E\Big|\int_0^T f(t)\,dW^\b (r)\Big|^2.
$$
In (iii) we used Lemma \ref{lem:estimate},
and (iv) follows from standard estimates for analytic semigroups (see \cite{Paz}).

Similarly, 
$$
\bal
\ &\E \Big\n \int_s^t S(t-r)B\,dW_\H^\b(r) \Big\n_{E_\theta}^2
\\ & \qquad \eqsim \E \Big\n \int_s^t (t-r)^{\b-\a} (\la-A)^{\theta}S(t-r)
(t-r)^{-\b+\a}B\,dW_\H^\b(r)
\Big\n^2
\\ & \qquad \le C^2 \E \Big\n \int_s^t (t-r)^{-\b+\a}B\,dW_\H^\b(r)\Big\n^2
\\ & \qquad = C^2 \n B\n_{\g(\H,E)}^2 \E\Big|\int_s^t
(t-r)^{-\b+\a}\,dW^\b(r)\Big|^2
\\ & \qquad \le C_T^2 \n B\n_{\g(\H,E)}^2 (t-s)^{2\a}.
\eal
$$
The first part of the theorem follows by combining these estimates.

For the second part, pick $\a<\a'<\b-\theta$.
Given $q\ge 1$, by the above we find a constant $C$ such that for $0\le s,t\le
T$,
$$\E \n U(t) - U(s)\n_{E_\theta}^q \le C^q |t-s|^{\a'q}. $$
For $q$ large enough
the existence of a version with $\a$-H\"older continuous trajectories now
follows from the Kolmogorov-Chentsov continuity theorem. Finally,
$ U \in L^p(\Om;C^\a([0,T];E_\theta))$ by Fernique's theorem \cite{Fer}.
\end{proof}

\section{An example}

We shall apply our results to prove existence and space-time H\"older regularity
of mild solutions for stochastic partial differential equations of the form
\beq\label{eq:scrA}
\left\{
\bal \frac{\partial u}{\partial t}(t,x) & = \mathscr{A}u(t,x) + \frac{\partial
W^\b(t,x)}{\partial t\,\partial x}, && x\in\OO, &&& t\in [0,T], \\
u(0,x) & = 0, && x\in \OO.
\eal
\right.
\eeq
where $\OO$ is a bounded $C^2$-domain in $\R^d$ and $\frac{\partial
W^\b(t,x)}{\partial t\partial x}$ denotes a space-time noise which is `white' in
space and `Liouville fractional' of order $0<\b<1$ in time.


We shall assume that $\mathscr{A}$ is a second-order uniformly elliptic operator
on $\OO$ of the form
$$\mathscr{A}f(x) = \sum_{i,j=1}^d a_{ij}(x) \frac{\partial^2 f}{\partial
x_i\partial x_j}(x) +  \sum_{j=1}^d b_{j}(x) \frac{\partial f}{\partial x_j}(x)
+ c(x)f(x).$$
The problem \eqref{eq:scrA} can be rewritten in the abstract form
\beq\label{eq:A}
\left\{
\bal dU(t) & = AU(t)\,dt  +\,dW_{L^2(\OO)}^\b(t), && t\in [0,T], \\
 U(0)& =0,
 \eal
 \right.  \eeq
where $W_{L^2(\OO)}^\b$ is an
$L^2(\OO)$-cylindrical fBm of order $\b$
on some probability space $(\Om,\F,\P)$

Under mild boundedness and regularity assumptions on the coefficients 
(to be precise, $a_{ij}\in C^\e(\overline{\OO})$ for some $\e>0$ and $b_{j}, c\in L^\infty(\OO)$), 
which we shall henceforth assume to be satisfied,
$A$ generates an analytic $C_0$-semigroup $S$ on $L^p(\OO)$. 
Moreover $A$ has bounded imaginary powers and hence, by \cite{Triebel}, 
the fractional domain spaces of $A$ are equal to the the complex interpolation spaces 
$(L^p(\OO))_\theta = \mathscr{D}((-A)^\theta)$
and given, up to equivalent norms, by 
$$ (L^p(\OO))_\theta = [L^p(\OO), \mathscr{D}(A)]_\theta
= \left\{
\begin{array}{ll}
H^{2\theta,p}(\OO),   & 0<\theta<\frac12, \\
H_0^{2\theta,p}(\OO), & \frac12<\theta<1.
\end{array}
\right.
$$

In the case $\b=\frac12$ the driving process $W_{L^2(\OO)}^\b$ is an
$L^2(\OO)$-cylindrical Brownian motion. In that case, in dimension $d=1$ the
mild solution $U$ of \eqref{eq:A} satisfies
$$ U \in L^q(\Om;C^\a([0,T];C^\g(\overline\OO)))$$
for all $1\le q<\infty$ and $\a,\g\ge 0$ satisfying  $2\a+\g<\frac{1}{2}$; see
\cite{Brz, DNW}. Following the methods used in these papers 
(where also more details can be found), 
for general
$0<\b<1$ we may apply Theorem \ref{thm:sACP2} in negative extrapolation spaces
of $L^p(\OO)$ of exponent greater than $\frac{d}{4}$. The regularising properties of the
semigroup $S$ then yield that problem \eqref{eq:A} admits a mild solution
$U$ in $L^q(\Om;C^\a([0,T];(L^p(\OO))_{\theta-\frac{d}{4}})) \subseteq
L^q(\Om;C^\a([0,T];H^{2\theta-\frac{d}{2},p}(\OO)))
$ for all $1\le q<\infty$, provided $\a,\theta\ge 0$ satisfy
$\frac{d}{4}<\theta<1$, $\theta\not=\frac12+\frac{d}{4}$, and $\a+\theta<\b.$
Combining this with the Sobolev embedding 
$H^{2\eta,p}(\OO) \embed C^\g(\overline\OO)$ for $\g+\frac{d}{p}< 2\eta$,
by taking $p$ large enough we obtain H\"older continuity of the solution jointly
in space and time:

\begin{theorem} Under the stated assumptions, the problem \eqref{eq:A} has mild
solution $U$ which belongs to  $L^q(\Om;C^\a([0,T];C^\g(\overline\OO)))$
for all $1\le q<\infty$ and $\a,\g\ge 0$ satisfying  $2\a+\g<2\b-\frac{d}{2}$.
\end{theorem}

In particular we see that a mild solution with space-time
H\"older regularity exists if $\frac{d}{4}<\b<1$. This contrasts the cylindrical
Brownian motion case $\b=\frac12$ where such solutions only exist in dimension
$d=1$. Note that in dimension $d=2$ and $d=3$ we obtain the existence of a space-time
H\"older continuous solution for $\frac12<\b<1$ and $\frac34<\b<1$, respectively.

\end{document}